\documentclass[a4paper,11pt]{article}

\usepackage{authblk}

\usepackage[utf8]{inputenc} % allow utf-8 input
\usepackage[T1]{fontenc}    % use 8-bit T1 fonts
\usepackage{indentfirst}
\usepackage{lmodern}

\usepackage{mathtools}
\usepackage{verbatim}
\usepackage{url}            % simple URL typesetting
\usepackage{booktabs}       % professional-quality tables
\usepackage{amsfonts}       % blackboard math symbols
\usepackage{amsmath}
\usepackage{amssymb}
\usepackage{amsthm}
\usepackage{microtype}      % microtypography
\usepackage[a4paper]{geometry} % Réduire les marges
\usepackage{enumitem}
\usepackage{xcolor}

\usepackage[linktocpage=true]{hyperref}
\hypersetup{
breaklinks=true,   % splits links across lines
colorlinks=true,   % displays links as colored text instead of blocks
pdfusetitle=true,  % \title and \author values into pdf metadata
}

\DeclareMathOperator{\Span}{span}
\DeclareMathOperator{\Dim}{dim}
\DeclareMathOperator{\Terms}{Terms}

\usepackage{graphicx}
\usepackage[makeroom]{cancel}
\usepackage{tikz}

\usepackage{hyperref}

\title{Independence, infinite dimension, and operators}
\author{Nizar El Idrissi and Samir Kabbaj}

\newcommand{\Addresses}{{% additional braces for segregating \footnotesize
  \bigskip
  \footnotesize

  \textbf{Nizar El Idrissi.}
  \par\nopagebreak Laboratoire : Equations aux dérivées partielles, Algèbre et Géométrie spectrales.
  \par\nopagebreak
  Département de mathématiques, faculté des sciences, université Ibn Tofail, 14000 Kénitra.\par\nopagebreak 
  \textit{E-mail address} : \texttt{nizar.elidrissi@uit.ac.ma}

  \medskip

  \textbf{Pr. Samir Kabbaj.} \par\nopagebreak Laboratoire : Equations aux dérivées partielles, Algèbre et Géométrie spectrales.
  \par\nopagebreak
  Département de mathématiques, faculté des sciences, université Ibn Tofail, 14000 Kénitra.\par\nopagebreak 
  \textit{E-mail address} : \texttt{samir.kabbaj@uit.ac.ma}
}}

\theoremstyle{plain}

\newtheorem{proposition}{Proposition}[section]
\newtheorem{corollary}{Corollary}[section]
\newtheorem{lemma}{Lemma}[section]

\theoremstyle{definition}
\newtheorem{definition}{Definition}[section]
\newtheorem{example}{Example}[section]   %% And a not so common one.

\theoremstyle{remark}

\makeatletter
\def\keywords{\xdef\@thefnmark{}\@footnotetext}
\makeatother

\begin{document}
\newpage
\maketitle
%%%%%%%%%%%%%%%%%%%%%%%%%%%
% abstract, keywords and Subject classification are optional.
%%%%%%%%%%%%%%%%%%%%%%%%%%%
\begin{abstract}
In [Appl. Comput. Harmon. Anal., 46(3):664–673, 2019] O. Christensen and M. Hasannasab observed that assuming the existence of an operator $T$ sending $e_n$ to $e_{n+1}$ for all $n \in \mathbb{N}$ (where  $(e_n)_{n \in \mathbb{N}}$ is a sequence of vectors) guarantees that $(e_n)_{n \in \mathbb{N}}$ is linearly independent if and only if $\dim \{e_n\}_{n \in \mathbb{N}} = \infty$. In this article, we recover this result as a particular case of a general order-theory-based model-theoretic result. We then return to the context of vector spaces to show that, if we want to use a condition like $T(e_i)=e_{\phi(i)}$ for all $i \in I$ where $I$ is countable as a replacement of the previous one, the conclusion will only stay true if $\phi : I \to I$ is conjugate to the successor function $succ : n \mapsto n+1$ defined on $\mathbb{N}$. We finally prove a tentative generalization of the result, where we replace the condition $T(e_i)=e_{\phi(i)}$ for all $i \in I$ where $\phi$ is conjugate to the successor function with a more sophisticated one, and to which we have not managed to find a new application yet.
\end{abstract}

% Most people don't use these, so they are "commented out"
% by starting the lines with a "%"
%\begin{keywords}
%   \LaTeX, typesetting
%\end{keywords}

%\begin{AMS}
%   50C60, 18C25
%\end{AMS}

%%%%%%%%%%%%%%%%%%%%%%
% % Here is the start of the Text
%%%%%%%%%%%%%%%%%%%%%%
\tableofcontents

\section{Introduction}

%% [ Literature search : I haven't found any characterization of the notion of linear independence in Google Scholar except "Nesterenko's linear independence criterion for vectors" and O. Christensen and M. Hasannab' criterion... ]

Linear algebra is an entrenched subject of mathematics that started with the introduction of coordinates in geometry by René Descartes. Its modern theory emerged in the late nineteenth century after Peano gave the definition of a vector space. This theory makes heavy use ot the concepts of linear independence and dimension, which often allow to state important theorems and conjectures.

Generally, linear independence of an infinite sequence implies that it spans an infinite-dimensional space, but not the opposite. As a result, it is interesting to consider the conditions of a reverse statement. Such reverse statements may allow to solve standing problems on linear independence.

In \cite{ChristensenHasannasab}, O. Christensen and M. Hasannasab observed that assuming the existence of an operator $T$ sending $e_n$ to $e_{n+1}$ for all $n \in \mathbb{N}$ (where  $(e_n)_{n \in \mathbb{N}}$ is a sequence of vectors) guarantees that $(e_n)_{n \in \mathbb{N}}$ is linearly independent if and only if $\dim \{e_n\}_{n \in \mathbb{N}} = \infty$. To wit:

\begin{proposition} (O. Christensen and M. Hasannasab) \\
\label{propChrisHasan}
Let $E$ be a vector space and $(e_n)_{n \in \mathbb{N}}$ a family in $E$ indexed by $\mathbb{N}$. Then 
\[ \left( \exists T \in L(\Span\{e_n\}_{n \in \mathbb{N}}, E) : \forall n \in \mathbb{N} : T(e_n)=e_{n+1} \right) \text{ and } \dim\Span\{e_n\}_{n \in \mathbb{N}} = +\infty ] \Rightarrow \]
\[ \quad (e_n)_{n \in \mathbb{N}} \text{ is free.} \] 
\end{proposition}

In this article, we prove some additional results related to proposition \ref{propChrisHasan}. First, we recover proposition \ref{propChrisHasan} as a particular case of a general order-theory-based model-theoretic result. We then return to the context of vector spaces to show that, if we want to use a condition like $T(e_i)=e_{\phi(i)}$ for all $i \in I$ where $I$ is countable as a replacement of the previous one, the conclusion will only stay true if $\phi : I \to I$ is conjugate to the successor function $succ : n \mapsto n+1$ defined on $\mathbb{N}$. We finally prove a tentative generalization of the result, where we replace the condition $T(e_i)=e_{\phi(i)}$ for all $i \in I$ where $\phi$ is conjugate to the successor function with a more sophisticated one, and to which we have not managed to find a new application yet. \\ \\
\noindent \textbf{Plan of the article.} We dedicate section \ref{sectionNotation} to the notations. We then prove in section \ref{sectionLemmasOrderTheory} some order-theoretic lemmas. These lemmas will allow us to prove in the next section \ref{sectionApplicationModelTheory} a model-theoretic result and recover proposition \ref{propChrisHasan} as a particular case. We then return in section \ref{sectionOnlyPossibleCountableExtensions} to the context of vector spaces and show that proposition \ref{propChrisHasan} can at most be generalized in the countable case to families $(e_i)_{i \in I}$ indexed by a countable set $I$ and maps $\phi : I \to I$ that are conjugate to the successor function $succ : n \mapsto n+1$ defined on $\mathbb{N}$, at least if we want to preserve a condition like $T(e_i)=e_{\phi(i)}$ for all $i \in I$. We finally prove in section \ref{sectionTentativeGeneralization} a tentative generalization of the result, where we replace the condition $T(e_i)=e_{\phi(i)}$ for all $i \in I$ where $\phi$ is conjugate to the successor function with a more sophisticated one.

\section{Notations}
\label{sectionNotation}

In the sequel, $\mathbb{N}$ denotes the set $\{0,1,2,\cdots\}$ of natural numbers including 0. $\mathbb{N}^*$ denotes $\mathbb{N} \setminus \{0\}$. \\
If $A$ is a set, we denote by $|A|$ the cardinality of $A$, $\mathcal{P}(A)$ the powerset of $A$, $\mathcal{P}_{\omega}(A)$ the set $\{B \subseteq A : |B| < \infty\}$, and $\mathcal{P}_{\omega,*}(A)$ the set $\{B \subseteq A : 0 < |B| < \infty\}$. \\
If $A$ is a set, $\phi : A \to A$ a self map and $n \in \mathbb{N}^*$, we denote by $\phi^n$ the composition of $\phi$ with itself $n$ times : $\phi \circ \cdots \circ \phi : A \to A$. In addition, we define $\phi^0$ to be the identity function on $A$. Moreover, if $a \in A$, we denote by $Orb_\phi(a)$ the forward orbit of $a$ under the iterates of $\phi$ : $\{\phi^n(a) : n \in \mathbb{N}\}$. \\
If $E$ and $F$ are two vector spaces, $L(E,F)$ denotes the set of linear operators from $E$ to $F$. When $E=F$, we simply write $L(E)$. If $E$ is a vector space, we denote by $L(*,E)$ the \textbf{class} $\{ T \in L(E^*,E) : E^* \text{ is a vector space}\}$.

\section{Lemmas in order theory}
\label{sectionLemmasOrderTheory}

\begin{definition}
Consider some set $P$ and a binary relation $\leq$ on $P$. Then $\leq$ is a \textit{preorder} if it is reflexive and transitive; i.e., for all $a$, $b$ and $c$ in $P$, we have that:
\begin{itemize}
\item $a \leq a$ (reflexivity)
\item if $a \leq b$ and $b \leq c$ then $a \leq c$ (transitivity)
\end{itemize}
A set that is equipped with a preorder is called a \textit{preordered set}.
\end{definition}

\begin{definition}
Consider a preordered set $(P,\le)$ and a map $p : P \to P$. Then $p$ is called a \textit{projection} if for all $a$ and $b$ in $P$, we have that:
\begin{itemize}
\item $a \le b$ implies $p(a) \le p(b)$ (p is monotone/increasing/order-preserving/isotone)
\item $p(p(a))=p(a)$ (idempotence)
\end{itemize}
\end{definition}

\begin{lemma}
\label{lemmaa0LEQp(b)ImpliesanLEQp(b)}
Let $(P,\leq)$ be a preordered set, $(a_n)_{n \in \mathbb{N}}$ a sequence in $P$, $b \in P$ and $p$ a projection on $P$. \\
Suppose there exists an increasing map $f : P \to P$ such that $f(p(b)) \le p(f(b))$, $\forall n \in \mathbb{N} : a_{n+1} \le f(a_n)$,  and $f(b) \le a_0$. \\
Then: $(a_0 \le p(b)) \Rightarrow (\forall n \in \mathbb{N} : a_n \le p(b))$.
\end{lemma}

\begin{proof}
Indeed, suppose that $a_0 \le p(b)$. Let's show by induction that $\forall n \in \mathbb{N} : a_n \le p(b)$. The base case is the hypothesis $a_0 \le p(b)$. Suppose we have that $a_n \le p(b)$ for some $n \in \mathbb{N}$. Then $a_{n+1} \le f(a_n) \le f(p(b)) \le p(f(b)) \le p(a_0) \le p(p(b)) = p(b)$. Hence $a_n \le p(b)$ holds for all $n \in \mathbb{N}$ and the lemma is proved.
\end{proof}

\begin{definition}
Consider some set $P$ and a binary relation $\leq$ on $P$. Then $\leq$ is a \textit{partial order} if it is a preorder and for all $a$ and $b$ in $P$, we have $a \leq b \leq a$ implies $a = b$ (antisymmetry). \\
A set that is equipped with a partial order is called a \textit{partially ordered set} or \textit{poset}.
\end{definition}

\begin{definition}
A set $S$ partially ordered by the binary relation $\leq$ is a \textit{join-semilattice} if for all elements $x$ and $y$ of $S$, the smallest upper bound of the set $\{x,y\}$ exists in $S$. \\
The smallest upper bound of the set $\{x,y\}$ is called the join of $x$ and $y$, denoted $x \vee y$. 
\end{definition}

\begin{lemma}
\label{lemmaam0LEQp(bm)ImpliesamnLEQp(bm)}
Let $(S,\le,\vee)$ be a join-semilattice, $p$ a projection on $S$, $(a_{m,n})_{(m,n) \in \mathbb{N}^2}$ a double sequence in $S$, $(b_m)_{m \in \mathbb{N}}$ an increasing sequence in $S$ such that $\forall m \in \mathbb{N} : a_{m,0} \leq p(b_{m+1})$. \\
Suppose there exists $m_\bullet \in \mathbb{N}$ and an increasing map $f : S \to S$ such that $f(p(b_{m_\bullet})) \le p(f(b_{m_\bullet}))$, $\forall (m,n) \in \mathbb{N}^2 : a_{m,n+1} \le f(a_{m,n})$, and $f(b_{m_\bullet}) \le \bigvee_{i=0}^{m_\bullet} a_{i,0}$. \\
Then: $(a_{m_\bullet,0} \le p(b_{m_\bullet}) \Rightarrow (\forall n \in \mathbb{N} : a_{m_\bullet,n} \le p(b_{m_\bullet}))$.
\end{lemma}

\begin{proof}
For all $(m,n) \in \mathbb{N}^2$, let $\widetilde{a_{m,n}} =  \bigvee_{i=0}^m a_{i,n}$. Suppose that $a_{m_\bullet,0} \le p(b_{m_\bullet})$. Then $\widetilde{a_{m_\bullet,0}} = \bigvee_{i=0}^{m_\bullet} a_{i,0} \le p(b_{m_\bullet})$ since $\forall i \in [\![0,m_\bullet-1]\!] : a_{i,0} \le p(b_{i+1})$, $a_{m_\bullet,0} \le p(b_{m_\bullet})$, and $(p(b_i))_{i \in \mathbb{N}}$ is increasing. Moreover, $f$ being increasing implies $\forall a,b \in S : f(a) \vee f(b) \le f(a \vee b)$, and so $\forall n \in \mathbb{N} : \widetilde{a_{m_\bullet,n+1}} = \bigvee_{i=0}^{m_\bullet} a_{i,n+1} \le \bigvee_{i=0}^{m_\bullet} f(a_{i,n}) \le f(\bigvee_{i=0}^{m_\bullet} a_{i,n}) = f(\widetilde{a_{m_\bullet,n}})$. Therefore, applying lemma \ref{lemmaa0LEQp(b)ImpliesanLEQp(b)} to $(\widetilde{a_{m_\bullet,n}})_{n \in \mathbb{N}}$, $b_{m_\bullet}$, and $f$, we have that $\forall n \in \mathbb{N} : \widetilde{a_{m_\bullet,n}} \le p(b_{m_\bullet})$ which implies $\forall n \in \mathbb{N} : a_{m_\bullet,n} \le p(b_{m_\bullet})$.
\end{proof}

\begin{lemma}
\label{lemmaemLEQp(bm)Impliesem+nLEQp(bm)}
Let $(S,\le,\vee)$ be a join-semilattice, $p$ a projection on $S$, $(e_n)_{n \in \mathbb{N}}$ and $(b_n)_{n \in \mathbb{N}}$ two sequences in $S$ such that $(b_n)_{n \in \mathbb{N}}$ is increasing and $\forall n \in \mathbb{N} : e_n \leq p(b_{n+1})$. \\
Suppose there exists $m_\bullet \in \mathbb{N}$ and an increasing map $f : S \to S$ such that $f(p(b_{m_\bullet})) \le p(f(b_{m_\bullet}))$, $\forall n \in \mathbb{N} : e_{n+1} \le f(e_n)$, and $f(b_{m_\bullet}) \le \bigvee_{i=0}^{m_\bullet} e_i$. \\
Then: $(e_{m_\bullet} \le p(b_{m_\bullet})) \Rightarrow (\forall n \in \mathbb{N} : e_{m_\bullet+n} \le p(b_{m_\bullet}))$.
\end{lemma}

\begin{proof}
Define $a_{m,n}$ as $e_{m+n}$ for all $(m,n) \in \mathbb{N}^2$. Then $((a_{m,n})_{(m,n) \in \mathbb{N}^2}, (b_n)_{n \in \mathbb{N}}, f)$ is a triple satisfying the conditions of lemma \ref{lemmaam0LEQp(bm)ImpliesamnLEQp(bm)} and so we deduce the result.
\end{proof}

\section{Application to model theory}
\label{sectionApplicationModelTheory}

In this section, we will use lemma \ref{lemmaemLEQp(bm)Impliesem+nLEQp(bm)} to prove a model-theoretic result. For a quick reference on model theory, see the book \cite{Hodges}. The model-theoretic result allows to recover proposition \ref{propChrisHasan} when we choose the $\sigma$-structure to be a vector space over a field. The idea of the following model-theoretic proposition is to consider an unsorted algebraic $\sigma$-structure $\mathcal{A} = (A,\sigma)$ and regard the map sending a set $X \subseteq A$ to the set $\Terms^{\mathcal{A}}[X]$ of interpreted terms with variables taken from $X$ as a special projection map on the boolean algebra $\mathcal{P}(A)$.

\begin{proposition}
Let $\sigma$ be an algebraic signature and $\mathcal{A} = (A,\sigma)$ a $\sigma$-structure. \\
Let $(e_n)_{n \in \mathbb{N}}$ be a sequence of elements in $A$. \\
Suppose there exists a map $f \in End_\sigma(A)$ such that $\forall n \in \mathbb{N} : f(e_n) = e_{n+1}$. \\
Then: $(\exists m \in \mathbb{N} : e_m \in \Terms^{\mathcal{A}}[e_0,\cdots,e_{m-1}]) \Rightarrow (\exists m \in \mathbb{N} : \forall n \in \mathbb{N} : e_{m+n} \in \Terms^{\mathcal{A}}[e_0,\cdots,e_{m-1}])$.
\end{proposition}

\begin{proof}
Notice that $(\mathcal{P}(A),\subseteq,\cup,\cap,\emptyset,A)$ is a boolean algebra and so a join-semilattice $(\mathcal{P}(A),\subseteq,\cup)$. Define $p$ as the map sending a set $X \subseteq A$ to the set $\Terms^{\mathcal{A}}[X]$ of interpreted terms with variables taken from $X$, and $b_n = \{e_0,\cdots,e_{n-1}\}$. The sequence $(e_n)_{n \in \mathbb{N}}$ in $A$ induces a sequence of singletons $(\{e_n\})_{n \in \mathbb{N}}$ in $\mathcal{P}(A)$. Also, the endomorphism $f$ induces the direct image map $f_\bullet = \begin{cases} \mathcal{P}(A) &\to \mathcal{P}(A) \\ X &\mapsto f(X) \end{cases}$ which is increasing, satisfies $\forall X \in \mathcal{P}(A) : f_\bullet(p(X)) = p(f_\bullet(X))$, $\forall i \in \mathbb{N} : \{e_{i+1}\} = f_\bullet(\{e_i\})$, and $f_\bullet(\{e_0,\cdots,e_{m-1}\}) = \{f(e_0),\cdots,f(e_{m-1})\} = \{e_1,\cdots,e_m\} \subseteq \bigcup_{i=0}^m \{e_i\}$. Moreover $(b_n)_{n \in \mathbb{N}}$ is increasing and $\forall n \in \mathbb{N} : \{e_n\} \subseteq \Terms^{\mathcal{A}}[e_0,\cdots,e_n]$. Suppose $\exists m_\bullet \in \mathbb{N} : e_{m_\bullet} \in \Terms^{\mathcal{A}}[e_0,\cdots,e_{m_\bullet-1}]$. Then the implication follows from lemma \ref{lemmaemLEQp(bm)Impliesem+nLEQp(bm)} applied to $(S,\le,\vee) := (\mathcal{P}(A),\subseteq,\cup)$, $p$, $(\{e_n\})_{n \in \mathbb{N}}$, $(b_n)_{n \in \mathbb{N}}$, $f_\bullet$, and $m_\bullet$.
\end{proof}

The previous proposition can be better appreciated after considering its contrapositive:

\begin{corollary}
\label{CorollaryEndExistsInfDimImpliesIndependence}
Let $\sigma$ be an algebraic signature and $\mathcal{A} = (A,\sigma)$ a $\sigma$-structure. \\
Let $(e_n)_{n \in \mathbb{N}}$ be a sequence of elements in $A$. \\
Suppose there exists a map $f \in End_\sigma(A)$ such that $\forall n \in \mathbb{N} : f(e_n) = e_{n+1}$. \\
Then: $(\forall m \in \mathbb{N} : \exists \varphi(m) \geq m : e_{\varphi(m)} \notin \Terms^{\mathcal{A}}[e_0,\cdots,e_{m-1}]) \Rightarrow (\forall m \in \mathbb{N} : e_m \notin \Terms^{\mathcal{A}}[e_0,\cdots,e_{m-1}])$.
\end{corollary}

\begin{example}
Consider the case of a vector space over a field, with its classical signature described for instance in \cite{Hodges} pp. 3-4. Then the corollary means that if we have an infinite sequence $(e_n)_{n \in \mathbb{N}}$ of vectors and a linear map with the property $f(e_n)=e_{n+1}$ for all $n \in \mathbb{N}$, then in order to show the linear independence of the sequence $(e_n)_{n \in \mathbb{N}}$, it is sufficient to prove that it spans an infinite-dimensional space. This is the result  appearing in the paper \cite{ChristensenHasannasab} by O. Christensen and M. Hasannasab (proposition \ref{propChrisHasan} of the present article).
\end{example}

\begin{corollary}
\label{corPropChrisHas}
Let $E$ be an infinite-dimensional vector space, $e \in E$, and $S$ a linear operator in $E$. Define the infinite sequence $(e_n)_{n \in \mathbb{N}}$ in $E$ as $e_n = S^n(e)$ for all $n \in \mathbb{N}$. Suppose that $\dim \Span (e_n)_{n \in \mathbb{N}} = \infty$. Then $(e_n)_{n \in \mathbb{N}}$ is free.
\end{corollary}

\begin{proof}
We have $\forall n \in \mathbb{N} : S(e_n) = e_{n+1}$ and we conclude by proposition \ref{propChrisHasan}.
\end{proof}

%% \begin{remark}
%% In general, it is impossible to generalize corollary \ref{corPropChrisHas} to multisequences of the form $e_{n_1,\cdots,n_r} = S_1^{n_1} \circ \cdots \circ S_r^{n_r}(e)$ where $S_1,\cdots,S_r$ are commuting operators since $\mathbb{N}^r$ doesn't admit a total order compatible with its additive law. 
%% \end{remark}

We do not know if it is possible to formulate a definition of infinite dimension or independence in the context of general algebraic $\sigma$-structures as in corollary \ref{CorollaryEndExistsInfDimImpliesIndependence}. If possible, corollary \ref{CorollaryEndExistsInfDimImpliesIndependence} may be advantageously applied to other algebraic $\sigma$-structures like groups, rings, algebras, etc., and have the same intuitive meaning of establishing a link between infinite dimension and independence.

\section{The only possible countable extensions}
\label{sectionOnlyPossibleCountableExtensions}

Before establishing the main proposition \ref{propCharacPEiPhiOrbPhi}, we need to recall some well-known lemmas.

\begin{lemma} 
\label{lemmaDistinct}
Let $I$ be an infinite set, $a \in I$ and $\phi : I \to I$. \\
Suppose that $Orb_\phi(a)$ is infinite. \\
Then $a, \phi(a), \phi(\phi(a)), \cdots$ are distinct.
\end{lemma}

\begin{proof}
We use euclidean division. Suppose that $\phi^{n}(a) = \phi^{m}(a)$ for $n < m$. By induction, we have $\phi^{n+j}(a) = \phi^{m+j}(a)$ for all $j \in \mathbb{N}$. Let $e \geq n$. Let $e - n = q(m-n) + r$ be the division with remainder of $e - n \in \mathbb{N}$ by $m-n \in \mathbb{N}^*$. If $q \geq 1$, we have $\phi^e(a) = \phi^{n + q(m-n) + r}(a) = \phi^{m + (q-1)(m-n) + r}(a) = \phi^{n + (q-1)(m-n) + r}(a)$. By immediate induction, we have that $\phi^e(a) = \phi^{n+r}(a)$, where $0 \leq r < m-n$. So $Orb_\phi(a) = \{a, \phi(a), \cdots, \phi^{m-1}(a)\}$ is finite, contradiction.
\end{proof}

\begin{lemma}
\label{lemmaTwoOrbCofAndInfExistPeriods}
Let $I$ be an infinite set, $(a,b) \in I^2$ and $\phi : I \to I$. \\
Suppose that $Orb_\phi(a)$ is infinite and $Orb_\phi(b)$ is cofinite. \\
Then $(\exists (m,n) \in \mathbb{N}^2) : \phi^m(a) = \phi^n(b)$.
\end{lemma}

\begin{proof}
Assume by way of contradiction that $(\forall (m,n) \in \mathbb{N}^2 : \phi^m(a) \neq \phi^n(b))$. Then $Orb_\phi(a) \subseteq I \setminus Orb_\phi(b)$. But this is impossible since $Orb_\phi(a)$ is infinite and $I \setminus Orb_\phi(b)$ is finite. Hence the result.
\end{proof}

\begin{lemma}
\label{lemmaOrbPhiIConjugate}
Let $I$ be a countably infinite set and $\phi : I \to I$. Then 
\begin{align*}
\left( \exists a \in I : Orb_\phi(a) = I \right) &\Leftrightarrow ( \phi \text{ is conjugate to } succ : \begin{cases} \mathbb{N} &\to \mathbb{N} \\ n &\mapsto n+1 \end{cases} \text{ in the sense that }\\
&\quad \exists \alpha : \mathbb{N} \to I \text{ such that } \alpha \text{ is bijective and } \alpha \circ succ = \phi \circ \alpha ).
\end{align*}
\end{lemma}

\begin{proof}
($\Rightarrow$) Suppose that $\exists a \in I : Orb_\phi(a) = I$. Define $\alpha : \begin{cases} \mathbb{N} &\to I \\ n &\mapsto \phi^n(a) \end{cases}$. By lemma \ref{lemmaDistinct}, $\alpha$ is bijective. Moreover, we have clearly $\alpha \circ succ = \phi \circ \alpha$. \\
($\Leftarrow$) Suppose $\phi$ is conjugated to $succ : \begin{cases} \mathbb{N} &\to \mathbb{N} \\ n &\mapsto n+1 \end{cases}$ by $\alpha$. Then we have $Orb_\phi(\alpha(0)) = \{\phi^n(\alpha(0))\}_{n \in \mathbb{N}} = \{\alpha(succ^n(0))\}_{n \in \mathbb{N}} = \{\alpha(n)\}_{n \in \mathbb{N}} = I$.
\end{proof}

The following proposition is the main result of this section.

\begin{proposition}
\label{propCharacPEiPhiOrbPhi}
Let $E$ be an infinite dimensional vector space and $I$ a countably infinite set. \\
For all $(e_i)_{i \in I} \in E^I$ and $\phi : I \to I$, let $P((e_i)_{i \in I},\phi)$ be the proposition 
\begin{align*}
(\exists T \in L(\Span(e_i)_{i \in I},E) : &\forall i \in I : T(e_i) = e_{\phi(i)} \text{ and } \dim\Span(e_i)_{i \in I} = +\infty )  \\
&\Rightarrow (e_i)_{i \in I} \text{ is free}. 
\end{align*}
Then we have
\[ \forall \phi : I \to I : \left[ \left( \forall (e_i)_{i \in I} \in E^I : P((e_i)_{i \in I},\phi) \right) \Leftrightarrow \exists a \in I : Orb_\phi(a) = I \right] \]
\end{proposition}

\begin{proof}
Let $\phi : I \to I$. \\
($\Rightarrow$) : Suppose that $\forall (e_i)_{i \in I} : P((e_i)_{i \in I},\phi)$. \\
Let's show first that all the orbits of $\phi$ are cofinite. Suppose by way of contradiction that $\exists a \in I : |I \setminus Orb_\phi(a)| = +\infty$. Let $(a_n)_{n \in \mathbb{N}}$ be a free family in $E$. We let $u : I \setminus Orb_\phi(a) \to \mathbb{N}$ be a bijection, and set for all $i \in I$, $e_i = \begin{cases} 0 \text{ if } i \in Orb_\phi(a) \\ a_{u(i)} \text{ otherwise} \end{cases}$. We define an operator on $\Span(e_i)_{i \in I}= \Span(a_n)_{n \in \mathbb{N}}$ by setting $T(a_n) = e_{\phi(u^{-1}(n))}$ for all $n \in \mathbb{N}$ and extending linearly. Then we have $\forall i \in I : T(e_i)=e_{\phi(i)}$. Indeed, if $i \in Orb_\phi(a)$, then $\phi(i) \in Orb_\phi(a)$ and so $e_i = e_{\phi(i)} = 0$ by definition which makes the relation true. Otherwise, we have $e_i = a_{u(i)}$ and so $T(e_i) = T(a_{u(i)}) = e_{\phi(u^{-1}(u(i)))} = e_{\phi(i)}$. Besides, we have $\dim\Span(e_i)_{i \in I} = +\infty$ since $\Span(e_i)_{i \in I} = \Span(a_n)_{n \in \mathbb{N}}$ and $\dim\Span(a_n)_{n \in \mathbb{N}} = +\infty$. Since $P((e_i)_{i \in I},\phi)$ is true, it follows that $(e_i)_{i \in I}$ is free which is impossible since $e_a = 0$. \\
Now take $a \in I$ such that $|I \setminus Orb_\phi(a) |$ is minimal. From lemma \ref{lemmaTwoOrbCofAndInfExistPeriods}, $\forall b \in I : \exists (m(b),n(b)) \in \mathbb{N}^2 : \phi^{m(b)}(b) = \phi^{n(b)}(a)$. If we choose $m(b)$ and $n(b)$ such that $m(b)+n(b)$ is minimal, then $n(b) \geq m(b)$ by minimality of $I \setminus Orb_\phi(a)$. Let $(a_n)_{n \in \mathbb{N}}$ be a free family in $E$. For all $i \in I$, we set $e_i = a_{n(i)-m(i)}$. We define an operator on $\Span(e_i)_{i \in I} = \Span(a_n)_{n \in \mathbb{N}}$ by setting $T(a_n) = a_{n+1}$ for all $n \in \mathbb{N}$ and extending linearly. Then we have $\forall i \in I : T(e_i) = e_{\phi(i)}$. Indeed, if $i \in Orb_\phi(a)$, then $m(i)=m(\phi(i))=0$ and $n(\phi(i)) = n(i)+1$ which implies $T(e_i)=T(a_{n(i)})=a_{n(i)+1}=a_{n(\phi(i))}=e_{\phi(i)}$. Otherwise, $m(i) \geq 1$, $m(\phi(i))=m(i)-1$ and $n(\phi(i))=n(i)$ since $\phi^{m(i)}(i) = \phi^{n(i)}(a)$ with $m(i)+n(i)$ minimal and $m(i) \geq 1$ implies $\phi^{m(i)-1}(\phi(i)) = \phi^{n(i)}(a)$ with $(m(i)-1)+n(i)$ minimal. So $T(e_i)=a_{n(i)-m(i)+1}=a_{n(i) - (m(i)-1)} = a_{n(\phi(i))-m(\phi(i))}=e_{\phi(i)}$. Besides, we have $\dim\Span(e_i)_{i \in I} = +\infty$ since $\Span(e_i)_{i \in I} = \Span(a_n)_{n \in \mathbb{N}}$ and $\dim\Span(a_n)_{n \in \mathbb{N}} = +\infty$. Since $P((e_i)_{i \in I},\phi)$ is true, it follows that $(e_i)_{i \in I}$ is free which implies $Orb_\phi(a) = I$ (otherwise, $\exists b \in I \setminus Orb_\phi(a)$, and so $e_b = a_{n(b)-m(b)} = e_{\phi^{n(b)}(a)}$ with $b \neq \phi^{n(b)}(a)$, contradicting the independence of $(e_i)_{i \in I}$). \\
($\Leftarrow$) : Suppose that $\exists a \in I : Orb_\phi(a) = I$. By lemma \ref{lemmaOrbPhiIConjugate}, $\phi$ is conjugated to $succ : \begin{cases} \mathbb{N} &\to \mathbb{N} \\ n &\mapsto n+1 \end{cases}$ by $\alpha$. Let $(e_i)_{i \in I} \in E^I$ and suppose there exists $T \in L(\Span(e_i)_{i \in I},E)$ such that for all $i \in I : T(e_i) = e_{\phi(i)}$ and $\dim\Span(e_i)_{i \in I} = +\infty$. Define $(f_n)_{n \in \mathbb{N}}$ by $f_n = e_{\alpha(n)}$ for all $n \in \mathbb{N}$. Then $(f_n)_{n \in \mathbb{N}}$ satisfies $\forall n \in \mathbb{N} : T(f_n)=f_{n+1}$ and $\dim\Span(f_n)_{n \in \mathbb{N}} = +\infty$ which by proposition \ref{propChrisHasan} implies that $(f_n)_{n \in \mathbb{N}}$ is free. Since $\alpha$ is a bijection, it follows that $(e_i)_{i \in I}$ is free.
\end{proof}

\section{A tentative generalization}
\label{sectionTentativeGeneralization}

\begin{definition}
Consider two sets $X$ and $Y$. A \textit{binary relation} $R$ on $X$ and $Y$ is a subset of the cartesian product $X \times Y$. The \textit{direct image} of a subset $S \subseteq X$ under a binary relation $R$ on $X$ and $Y$ is written $R[S]$ and refers to $\{y \in Y : \exists x \in S : (x,y) \in R\}$. If $X=Y$, a binary relation on $X$ and $Y$ is simply called a relation on $X$. The relation $\Delta_X := \{(x,x) : x \in X\}$ is called the \textit{identity relation} on $X$. The \textit{composition} of two binary relations $R_1$ and $R_2$ over $X$ and $Y$, and $Y$ and $Z$ (respectively), is the binary relation over $X$ and $Z$, denoted $R_1 \circ R_2$, and given by the subset $\{(x,z) \in X \times Z : \exists y \in Y : (x,y) \in R_1 \wedge (y,z) \in R_2\}$. If $R$ is a binary relation on $X$, we define $R^0$ as $\Delta_X$ and for all $n \geq 1$, $R^n$ as $R \circ R^{n-1}$.
\end{definition}

The following proposition which uses the language of functions with set-valued inputs or/and outputs is the main result of this section.

\begin{proposition}
\label{prop1}
Let $E$ be an infinite dimensional vector space, $V$ a finite dimensional subspace of $E$, $(e_i)_{i \in I}$ a family of vectors in $E$, and $J$ an infinite set. Suppose that
\begin{enumerate}
\item There exist two functions $u : \mathcal{P}_{\omega,*}(I) \to I$ and $G : \mathcal{P}_{\omega,*}(I) \to \mathcal{P}_{\omega,*}(I)$ such that for all $I^* \in \mathcal{P}_{\omega,*}(I)$, $u(I^*) \in I^*$ and $I^* \subseteq G(I^*)$,
\item There exists a function $T : \mathcal{P}_{\omega,*}(I) \times J \to L(*,E)$ such that for all $(I^*,j) \in \mathcal{P}_{\omega,*}(I) \times J$, $T(I^*,j) \in L(\Span(e_i)_{i \in G(I^*)},E)$,
\item There exists a finite subset $J_0$ of $J$ and a relation $R$ on $J$ such that $R[j] \subseteq J_0$ for all $j \in J_0$ and $(\forall j \in J)(\exists n_j \in \mathbb{N}) R^{n_j}[j] \subseteq J_0$,
\item $\forall i \in I : \Dim\Span\{T(I^*,j) e_i\}_{j \in J} = \infty$,
\item $\forall (I^*,j) \in \mathcal{P}_{\omega,*}(I) \times J : \forall i \in G(I^*) \setminus \{u(I^*)\}:$
\[T(I^*,j)e_i \in \Span\{T(I^*,j')e_{i'}\}_{(j',i')\in R[j] \times G(I^*)} + V.\]
\end{enumerate}
Then $(e_i)_{i \in I}$ is free.
\end{proposition}

\begin{proof} (of proposition \ref{prop1})
Assume by way of contradiction that $(e_i)_{i \in I}$ is dependent. Then, there exists $I^* \in \mathcal{P}_{\omega,*}(I)$ and coefficients $\{c_i\}_{i \in I^*}$ in $\mathbb{C}^*$ such that 
\[ \sum_{i \in I^*} c_i e_i = 0, \] 
which implies
\[ e_{u(I^*)} = \sum_{i \in I^* \setminus \{u(I^*)\}} \frac{-c_i}{c_{u(I^*)}} e_i. \]
By linearity of $T(I^*,j)$, we then have for all $j \in J$
\begin{equation}
\label{equation11}
T(I^*,j)e_{u(I^*)} = \sum_{i \in I^* \setminus \{u(I^*)\}} \frac{-c_i}{c_{u(I^*)}} T(I^*,j) e_i. \tag{$*$}
\end{equation}
We will prove by induction on $n \in \mathbb{N}$ that 
\begin{equation}
\label{equation12}
\forall (n,j,i) \in \mathbb{N} \times J \times (G(I^*) \setminus \{u(I^*)\}) : T(I^*,j)e_i \in \Span \{T(I^*,j')e_i'\}_{(j',i')\in R^n[j] \times G(I^*)} + V. \tag{$**$}
\end{equation}
\begin{itemize}
\item For $n=0$, the relation is true since $R^0(j)=\Delta_J[j] = \{j\}$ and $i \in G(I^*)$ (take $(j',i'):=(j,i)$ and $v := 0$).
\item Suppose the induction hypothesis is true at the order $n \in \mathbb{N}$. So for all $(j,i) \in J \times (G(I^*) \setminus \{u(I^*)\})$, there exist complex coefficients $\{\alpha_{j',i'}^{j,i}\}_{(j',i')\in R[j] \times G(I^*)}$ and $\{\beta_{j',i'}^{j,i}\}_{(j',i')\in R^n[j] \times G(I^*)}$, and vectors $v^{j,i}$ and $w^{j,i}$ in $V$ such that 
\begin{equation}
\label{equation13}
T(I^*,j)e_i = \sum_{(j',i') \in R[j] \times G(I^*)} \alpha_{j',i'}^{j,i} T(I^*,j')e_{i'} + v^{j,i}, \tag{$***$}
\end{equation}
and 
\begin{equation}
\label{equation14}
T(I^*,j)e_i = \sum_{(j',i') \in R^n[j] \times G(I^*)} \beta_{j',i'}^{j,i} T(I^*,j')e_{i'} + w^{j,i}, \tag{$****$}
\end{equation}
because of condition 5 and the induction hypothesis. \\
Let $(j,i) \in J \times (G(I^*) \setminus \{u(I^*)\})$. We have
\begin{align*}
T(I^*,j)e_i &= \sum_{(j',i') \in R^n[j] \times G(I^*)} \beta_{j',i'}^{j,i} T(I^*,j')e_{i'} + w^{j,i} \\
&= \sum_{(j',i') \in R^n[j] \times (G(I^*) \setminus \{u(I^*)\}} \beta_{j',i'}^{j,i} T(I^*,j')e_{i'} \\
&\quad + \sum_{j' \in R^n[j]} \beta_{j',u(I^*)}^{j,i} T(I^*,j')e_{u(I^*)} + w^{j,i} \\
&= \sum_{(j',i') \in R^n[j] \times (G(I^*) \setminus \{u(I^*)\}} \beta_{j',i'}^{j,i} T(I^*,j')e_{i'} \\
& \quad \quad + \sum_{(j',i') \in R^n[j] \times (G(I^*) \setminus \{u(I^*)\})} \beta_{j',u(I^*)}^{j,i} \frac{-c_{i'}}{c_{u(I^*)}}T(I^*,j')e_{i'} + w^{j,i},
\end{align*}
where we have used equation \eqref{equation11} and extended $\{c_i\}_{i \in I^*}$ to $\{c_i\}_{i \in G(I^*)}$ by setting $\forall i' \in G(I^*)\setminus I^* : c_{i'}=0$.
Simplifying, we have
\begin{align*}
T(I^*,j)e_i &= \sum_{(j',i') \in R^n[j] \times (G(I^*) \setminus \{u(I^*)\}} (\beta_{j',i'}^{j,i} - \beta_{j',u(I^*)}^{j,i} \frac{c_{i'}}{c_{u(I^*)}}) T(I^*,j')e_{i'} + w^{j,i}.
\end{align*}
Using equality \eqref{equation13} for each term $T(I^*,j')e_{i'}$ and then rearranging the resulting sum, we have that $T(I^*,j)e_i$ is equal to
\begin{align*}
&\sum_{(j'',i'') \in R^{n+1}[j] \times G(I^*)} \left[ \sum_{(j',i') \in F(j'',j,n) \times G(I^*) \setminus \{u(I^*)\}} \alpha_{j'',i''}^{j',i'} \left(\beta_{j',i'}^{j,i} - \beta_{j',u(I^*)}^{j,i} \frac{c_{i'}}{c_{u(I^*)}} \right) \right] \\
&\cdot T(I^*,j'')f_{i''} + \left[ \sum_{(j',i') \in R^n[j] \times (G(I^*) \setminus \{u(I^*)\}} (\beta_{j',i'}^{j,i} - \beta_{j',u(I^*)}^{j,i} \frac{c_{i'}}{c_{u(I^*)}})v^{j',i'} + w^{j,i} \right] \\
&\in \Span\{T(I^*,j')e_{i'}\}_{(j',i')\in R^{n+1}[j] \times G(I^*)} + V,
\end{align*}
where $F(j'',j,n) := \{j' \in R^n[j] : j'' \in R[j'] \}$.
\end{itemize}
Hence the claim is proved. Now letting $j \in J$ and using condition 3, there exists $n_j \in \mathbb{N}$ such that $R^{n_j}[j] \subseteq J_0$, which implies by \eqref{equation11} and the claim \eqref{equation12} that 
\[ T(I^*,j)e_{u(I^*)} \in \Span\{T(I^*,j')e_{i'}\}_{(j',i')\in J_0 \times G(I^*)} + V. \]
Therefore
\[ \Span\{T(I^*,j)e_{u(I^*)}\}_{j \in J} \subseteq \Span\{T(I^*,j')e_{i'}\}_{(j',i')\in J_0 \times G(I^*)} + V. \]
Since $\Span\{T(I^*,j')e_{i'}\}_{(j',i')\in J_0 \times G(I^*)} + V$ is finite-dimensional, this contradicts condition 4.
\end{proof}

\begin{example} (from O. Christensen and M. Hasannab \cite{ChristensenHasannasab}) \label{exChrisHasan} \\
Let $E$ be an infinite-dimensional vector space and $(e_n)_{n \in \mathbb{N}}$ a sequence in $E$. \\ Suppose that $\Dim\Span(e_n)_{n \in \mathbb{N}} = \infty$ and there exists an operator $S : \Span(e_n)_{n \in \mathbb{N}} \to E$ such that $S(e_n) = e_{n+1}$ for all $n \in \mathbb{N}$. Then $(e_n)_{n \in \mathbb{N}}$ is free. \\
Indeed, define $u : \mathcal{P}_{\omega,*}(\mathbb{N}) \to \mathbb{N}$ and $G : \mathcal{P}_{\omega,*}(\mathbb{N}) \to \mathcal{P}_{\omega,*}(\mathbb{N})$ by $u(I^*) = \max(I^*)$ and $G(I^*) = [\![0,\max(I^*)]\!]$ for all $I^* \in \mathcal{P}_{\omega,*}(\mathbb{N})$. Moreover, define $T : \mathcal{P}_{\omega,*}(\mathbb{N}) \times \mathbb{N} \to L(*,E)$ by $T(I^*,m) = (S^m)_{|\Span(e_n)_{n \in [\![0,\max(I^*)]\!]}}$ for all $(I^*,m) \in \mathcal{P}_{\omega,*}(\mathbb{N}) \times \mathbb{N}$. Take $J_0 = \{0\}$ and define $R =  (0,0) \bigcup \left( \bigcup_{m \in \mathbb{N}^*} (m,m-1) \right) \subseteq \mathbb{N} \times \mathbb{N}$. Examining the conditions of proposition \ref{prop1}, we see that conditions 1-2 are valid, 3 is true because $R^m[m] = \{0\} = J_0$ for all $m \in \mathbb{N}$, 4 is valid because $\{T(I^*,m)e_{u(I^*)}\}_{m \in \mathbb{N}} = \{e_{u(I^*)+m}\}_{m \in \mathbb{N}}$ spans an infinite dimensional space since $(e_m)_{m \in \mathbb{N}}$ does, and finally condition 5 is also valid since for all $(I^*,m) \in \mathcal{P}_{\omega,*}(\mathbb{N}) \times \mathbb{N}^*$ and for all $n \in [\![0,\max(I^*)-1]\!]$ we have $T(I^*,m)e_n = e_{n+m} = T(I^*,m-1)e_{n+1}$ (notice that $m-1 \in R[m]$ and $n+1 \in [\![0,\max(I^*)]\!]$), and for all $I^* \in \mathcal{P}_{\omega,*}(\mathbb{N})$ and $n \in [\![0,\max(I^*)-1]\!]$ we have $T(I^*,0)e_n = e_n = T(I^*,0)e_n$ (notice that $0 \in R[0]$ and $n \in [\![0,\max(I^*)]\!]$). Hence the result.
\end{example}

Before we move on to the next example, we need to prove two additional lemmas.

\begin{lemma}
\label{lemmaOrbCofAlmostAllOrbCof}
Let $I$ be an infinite set, $(a,b) \in I^2$ and $\phi : I \to I$. \\
Suppose that $Orb_\phi(a)$ is cofinite. Then
\begin{enumerate}
\item If $\exists n \in \mathbb{N} : b = \phi^n(a)$, then $Orb_\phi(b)$ is cofinite,
\item If $\exists n \in \mathbb{N} : a = \phi^n(b)$, then $Orb_\phi(b)$ is cofinite.
\end{enumerate}
\end{lemma}

\begin{proof}
First, notice that by lemma \ref{lemmaDistinct}, $a, \phi(a), \phi^2(a), \cdots$ are distinct. \\
$(1) :$ Suppose $\exists n \in \mathbb{N} : b = \phi^n(a)$. Then $I \setminus Orb_\phi(b) = I \setminus \{\phi^m(a)\}_{m \geq n} = \left( I \setminus Orb_\phi(a) \right) \cup \{\phi^m(a)\}_{m < n}$ is finite. \\
$(2) :$ Suppose $\exists n \in \mathbb{N} : a = \phi^n(b)$. Then $I \setminus Orb_\phi(b) = \left( I \setminus \{\phi^m(b)\}_{m \geq n} \right) \setminus \{\phi^m(b)\}_{m < n} = \left( I \setminus Orb_\phi(a) \right) \setminus \{\phi^m(b)\}_{m < n}$ is finite.
\end{proof}

\begin{corollary}
\label{corollaryOrbPhiIAllOrbCof}
Let $I$ be an infinite set, $a \in I$ and $\phi : I \to I$. \\
Suppose that $Orb_\phi(a) = I$. \\
Then $\forall b \in I : Orb_\phi(b)$ is cofinite
\end{corollary}

\begin{lemma}
\label{lemmaCharacOrbPhiIAlmostPreservation}
Let $I$ be an infinite set and $\phi : I \to I$ such that $\forall a \in I : Orb_\phi(a)$ is infinite. Then we have
\begin{align*} 
&\left( \exists a \in I : Orb_\phi(a) = I  \right) \Rightarrow \\
&\quad \quad (\text{there exist two functions } u : \mathcal{P}_{\omega,*}(I) \to I \text{ and } G : \mathcal{P}_{\omega,*}(I) \to \mathcal{P}_{\omega,*}(I) \text{ such that} \\
&\quad \quad \forall I^* \in \mathcal{P}_{\omega,*}(I) : u(I^*) \in I^*, I^* \subseteq G(I^*), \text{ and } \phi(G(I^*) \setminus \{u(I^*)\}) \subseteq G(I^*) ).
\end{align*}
\end{lemma}

\begin{proof}
Define
\[ n : \begin{cases} \mathcal{P}_{\omega,*}(I) &\to \mathbb{N} \\
                       I^* &\mapsto \max\{n \in \mathbb{N} : \phi^n(a) \in I^* \}
\end{cases} \]
which is well-defined because $\{a,\phi(a),\phi^2(a),\cdots\}$ are distinct (lemma \ref{lemmaDistinct}) and let 
\[ u : \begin{cases} \mathcal{P}_{\omega,*}(I) &\to I \\
                       I^* &\mapsto \phi^{n(I^*)}(a)
\end{cases} \]
and 
\[ G : \begin{cases} \mathcal{P}_{\omega,*}(I) &\to \mathcal{P}_{\omega,*}(I) \\
                       I^* &\mapsto \{a, \phi(a), \cdots, \phi^{n(I^*)}(a)\} 
\end{cases}. \]
With this choice, since $Orb_\phi(a)=I$, $u$ and $G$ do satisfy the requirements.
\end{proof}

\begin{example}
Let $E$ be an infinite dimension vector space, $I$ a countably infinite set, $(e_i)_{i \in I}$ a family in $E$ and $\phi : I \to I$ such that $\exists a \in I : Orb_\phi(a) = I$. \\
Suppose that $\Dim\Span\{e_i\}_{i \in I} = \infty$ and there exists an operator $S : \Span\{e_i\}_{i \in I} \to E$ such that $S(e_i) = e_{\phi(i)}$ for all $i \in I$. Then $\{e_i\}_{i \in I}$ is free. \\
Indeed, from lemma \ref{lemmaCharacOrbPhiIAlmostPreservation}, there exist two functions $u : \mathcal{P}_{\omega,*}(I) \to I$ and $G : \mathcal{P}_{\omega,*}(I) \to \mathcal{P}_{\omega,*}(I)$ such that for all $I^* \in \mathcal{P}_{\omega,*}(I)$, $u(I^*) \in I^*$, $I^* \subseteq G(I^*)$ and $\phi(G(I^*) \setminus \{u(I^*)\}) \subseteq G(I^*)$. Define $T : \mathcal{P}_{\omega,*}(I) \times \mathbb{N} \to L(*,E)$ by $T(I^*,m) = (S^m)_{|\Span\{e_i\}_{i \in G(I^*)}}$ for all $(I^*,m) \in \mathcal{P}_{\omega,*}(I) \times \mathbb{N}$. Take $J_0 = \{0\}$ and define $R =  (0,0) \bigcup \left( \bigcup_{m \in \mathbb{N}^*} (m,m-1) \right) \subseteq \mathbb{N} \times \mathbb{N}$. Examining the conditions of proposition \ref{prop1}, we see that conditions 1-2 are valid, 3 is true because $R^m[m] = \{0\} = J_0$ for all $m \in \mathbb{N}$, 4 is valid because $\{T(I^*,m)e_{u(I^*)}\}_{m \in \mathbb{N}} = \{e_{\phi^m(u(I^*))}\}_{m \in \mathbb{N}}$ spans an infinite-dimensional space since $(e_i)_{i \in I}$ does and $|I \setminus \{\phi^m(u(I^*))\}_{m \in \mathbb{N}}| < \infty$ (corollary \ref{corollaryOrbPhiIAllOrbCof}), and finally condition 5 is also valid since for all $(I^*,m) \in \mathcal{P}_{\omega,*}(I) \times \mathbb{N}^*$ and for all $i \in G(I^*) \setminus \{u(I^*)\}$ we have $T(I^*,m)e_i = e_{\phi^m(i)} = T(I^*,m-1)e_{\phi(i)}$ (notice that $m-1 \in R[m]$ and $\phi(i) \in G(I^*)$), and for all $I^* \in \mathcal{P}_{\omega,*}(I)$ and $i \in G(I^*) \setminus \{u(I^*)\}$ we have $T(I^*,0)e_i = e_i = T(I^*,0)e_i$ (notice that $0 \in R[0]$ and $i \in G(I^*)$). Hence the result.
\end{example}

\section*{Acknowledgement}
The first author is financially supported by the \textit{Centre National pour la Recherche Scientifique et Technique} of Morocco.

\section*{Conflict of interest}
On behalf of all authors, the corresponding author states that there is no conflict of interest.

\nocite{*}
\bibliographystyle{plain}
\bibliography{references}

\Addresses

\end{document}